\documentclass{amsart}

\newtheorem*{Theorem}{Theorem}

\newtheorem{Claim}{Claim}

\theoremstyle{remark}

\begin{document}

\title{Cosmetic surgeries and non-orientable surfaces}

\author{Kazuhiro Ichihara}
\address{Department of Mathematics, 
College of Humanities and Sciences, Nihon University,
3-25-40 Sakurajosui, Setagaya-ku, Tokyo 156-8550, Japan}
\email{ichihara@math.chs.nihon-u.ac.jp}
\thanks{The first author is partially supported by
Grant-in-Aid for Young Scientists (B), No.~23740061, 
Ministry of Education, Culture, Sports, Science and Technology, Japan.}

\begin{abstract}
By considering non-orientable surfaces in the surgered manifolds, 
we show that the $10/3$- and $-10/3$-Dehn surgeries 
on the 2-bridge knot $9_{27} = S(49,19)$ are not cosmetic, 
i.e., they give mutually non-homeomorphic manifolds. 
The knot is unknown to have no cosmetic surgeries 
by previously known results; in particular, 
by using the Casson invariant and the Heegaard Floer homology. 
\end{abstract}



\date{\today}

\maketitle

\section{Introduction}

The well-known Knot Complement Conjecture says that: 
If two knots in the 3-sphere $S^3$ have homeomorphic complements, 
then they are equivalent, 
i.e., there exists a homeomorphism $h : S^3 \to S^3$ 
which takes one knot to the other. 
It had been conjectured by Tietze in 1908 \cite{Tietze}, and 
was proved by Gordon and Luecke 
in their cerebrated paper \cite{GordonLuecke} in 1989. 
Actually Gordon and Luecke showed that; 
On a nontrivial knot in $S^3$,
nontrivial \textit{Dehn surgery} never yields $S^3$; 
to which the Knot Complement Conjecture is an immediate corollary. 

The Knot Complement Conjecture can be generalized as follows. 

\medskip

\noindent
\textbf{Oriented Knot Complement Conjecture} 
(Bleiler (Kirby's list Problem 1.81(D) \cite{Kirby})) : 
If $K_1$ and $K_2$ are knots in a closed, oriented 
3-manifold $M$ whose complements are
homeomorphic via an orientation-preserving 
homeomorphism, then there exists 
an orientation-preserving homeomorphism of $M$ 
taking $K_1$ to $K_2$.

\medskip

This conjecture is equivalent to the following in terms of Dehn surgery. 

\medskip

\noindent
\textbf{Cosmetic Surgery Conjecture} 
(Bleiler (Kirby's list Problem 1.81(A) \cite{Kirby})): 
Two surgeries on inequivalent slopes are never purely \textit{cosmetic}. 

\medskip

Here we say that: 
two slopes are \textit{equivalent} 
if there exists a homeomorphism of the exterior $E(K)$ of a knot $K$ 
taking one slope to the other, and 
two surgeries on $K$ along slopes $r_1$ and $r_2$ 
are \textit{purely cosmetic} 
if there is an orientation preserving 
homeomorphism between $K(r_1)$ and $K(r_2)$, 
and \textit{chirally cosmetic} 
if the homeomorphism is orientation reversing.

Toward Cosmetic Surgery Conjecture by topological approach, 
as a first step, we show the following in this paper: 

\begin{Theorem}
Let $K$ be the knot in $S^3$ indicated as $9_{27}$ in the Rolfsen's knot table, 
which is a two-bridge knot with the Schubert form $S(49,19)$. 
Then $K (10/3)$ is not homeomorphic to $K(-10/3)$. 
\end{Theorem}

We remark that the slopes corresponding to $10/3$ and $-10/3$ 
are inequivalent. 
Because, if they were, there must be 
an orientation reversing homeomorphism on $E(K)$ 
by \cite[Lemma 2]{BleilerHodgsonWeeks}, 
but it is impossible since the knot is not amphicheiral. 

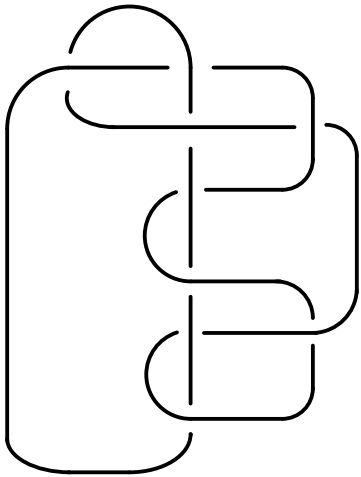
\begin{figure}[hbt]
\unitlength 0.1in
\begin{picture}( 18.3000, 24.4000)( 46.0000,-34.4000)
%
\special{pn 20}%
\special{pa 5560 1320}%
\special{pa 5560 1552}%
\special{fp}%
%
\special{pn 20}%
\special{ar 5240 1320 320 320  3.3989164 6.2831853}%
%
\special{pn 20}%
\special{pa 5440 1320}%
\special{pa 4920 1320}%
\special{fp}%
%
\special{pn 20}%
\special{ar 4920 1640 320 320  3.1415927 4.7123890}%
%
\special{pn 20}%
\special{pa 4600 1640}%
\special{pa 4600 3276}%
\special{fp}%
%
\special{pn 20}%
\special{ar 5160 1480 248 152  1.5707963 3.3482759}%
%
\special{pn 20}%
\special{pa 5144 1632}%
\special{pa 6104 1632}%
\special{fp}%
%
\special{pn 20}%
\special{pa 5680 1320}%
\special{pa 6040 1320}%
\special{fp}%
%
\special{pn 20}%
\special{ar 6040 1480 160 160  4.7123890 6.2831853}%
%
\special{pn 20}%
\special{ar 6040 1800 160 160  6.2831853 6.2831853}%
\special{ar 6040 1800 160 160  0.0000000 1.5707963}%
%
\special{pn 20}%
\special{pa 6200 1800}%
\special{pa 6200 1480}%
\special{fp}%
%
\special{pn 20}%
\special{pa 5560 1744}%
\special{pa 5560 2360}%
\special{fp}%
%
\special{pn 20}%
\special{pa 6040 1960}%
\special{pa 5640 1960}%
\special{fp}%
%
\special{pn 20}%
\special{ar 6270 1780 160 160  4.7123890 6.2831853}%
%
\special{pn 20}%
\special{ar 5560 2200 240 240  1.5707963 4.3906384}%
%
\special{pn 20}%
\special{pa 5560 2440}%
\special{pa 6030 2440}%
\special{fp}%
%
\special{pn 20}%
\special{pa 5560 2520}%
\special{pa 5560 3080}%
\special{fp}%
%
\special{pn 20}%
\special{pa 6430 1780}%
\special{pa 6430 2498}%
\special{fp}%
%
\special{pn 20}%
\special{ar 6196 2476 234 234  0.0471349 1.5137155}%
%
\special{pn 20}%
\special{pa 6220 2710}%
\special{pa 5630 2710}%
\special{fp}%
%
\special{pn 20}%
\special{ar 6008 2632 192 192  4.7123890 6.2831853}%
%
\special{pn 20}%
\special{ar 5560 2928 232 232  1.5707963 4.4003469}%
%
\special{pn 20}%
\special{pa 5560 3160}%
\special{pa 6040 3160}%
\special{fp}%
%
\special{pn 20}%
\special{ar 6040 3000 160 160  6.2831853 6.2831853}%
\special{ar 6040 3000 160 160  0.0000000 1.5707963}%
%
\special{pn 20}%
\special{pa 6200 3000}%
\special{pa 6200 2776}%
\special{fp}%
%
\special{pn 20}%
\special{ar 5240 3240 320 200  6.2831853 6.2831853}%
\special{ar 5240 3240 320 200  0.0000000 1.5707963}%
%
\special{pn 20}%
\special{pa 5240 3440}%
\special{pa 4920 3440}%
\special{fp}%
%
\special{pn 20}%
\special{ar 4928 3264 328 176  1.5707963 3.1415927}%
\end{picture}%
\caption{The knot $9_{27}$}
\end{figure}

This is very specific calculations for just one example, 
but previously known results about cosmetic surgery 
cannot distinguish the pair of manifolds. 
Also, as far as the author knows, 
there are no approaches toward Cosmetic Surgery Conjecture 
by considering non-orientable surfaces, and so, 
our arguments could give a something new viewpoint. 

\medskip

We here explain our theorem above is in fact contained 
in the complement of the recent known results, 
mainly based on \textit{Heegaard Floer technology}, 
developed and mainly studied by P. Ozsv\'{a}th and Z. Szab\'{o}. 

Originally such an approach had started in \cite{OzsvathSzabo11}. 
After that, together with other invariants of 3-manifolds, 
Wang showed in \cite{Wang} that 
no genus one knot in $S^3$ admits purely cosmetic surgeries. 

Moreover, in \cite{Wu}, Wu proved that 
for two distinct rational numbers $r$ and $r'$ with $r r' > 0$ 
and a non-trivial knot $K$ in $S^3$, 
$K(r)$ is not orientation preservingly homeomorphic to $K(r')$. 

The key ingredient of Wu's proof is using 
the Casson invariant of 3-manifolds introduced by A. Casson. 
Actually, Boyer and Lines in \cite{BoyerLines} 
had previously proved by using the Casson invariant 
that a knot $K$ in $S^3$ satisfying $\Delta''_K (1) \ne 0$ has no cosmetic surgeries. 
Here $\Delta_K (t)$ denotes the Alexander polynomial for $K$. 

Recently Ni and Wu obtained the following excellent result in \cite{NiWu}; 
Suppose $K$ is a nontrivial knot in $S^3$, 
$r_1, r_2 \in \mathbb{Q} \cup \{ 0/1 \}$ are two distinct slopes 
such that $K(r_1) \cong K (r_2)$ as oriented manifolds. 
Then $r_1, r_2$ satisfy that
(a) $r_1 = - r_2$, 
(b) $q^2 \equiv -1 \mod p$ for $r_1 = p/q$, 
(c) $\tau (K) = 0$, 
where $\tau$ is the invariant defined by Ozsv\'{a}th-Szab\'{o} defined in \cite{OzsvathSzabo03}. 

On the other hand, 
$K=9_{27}$ is a slice knot, and so $|\tau(K)| \le g_4(K) =0$ 
by \cite[Corollary 1.3]{OzsvathSzabo03}. 
Actually $K$ is a 2-bridge knot, and so, is an alternating knot, 
for which $\tau(K) = - \sigma(K) / 2 $ holds, 
where $\sigma(K)$ denotes the knot signature, 
as shown in \cite[Theorem 1.4]{OzsvathSzabo03}. 
Obviously we see that $q^2 = 9 \equiv -1	\mod p=10$.

Moreover the Alexander polynomial for $K=9_{27}$ is
$\Delta_K (t) = -t^3 + 5 t^2 -11t + 15 -11t^{-1} + 5 t^{-2} - t^{-3}$, 
which implies $ \Delta''_K (1) = 0$.

\section{Known facts} 

We here summarize part of known results on cosmetic surgery conjecture. 
See \cite{BleilerHodgsonWeeks} in detail. 

We first remark that 
the Cosmetic surgery conjecture for ``chirally cosmetic'' case is not true: 
there exists counter-example given by Mathieu \cite{Mathieu90, Mathieu92}. 
Actually, for example, 
$(18k+9)/(3k+1)$- and $(18k+9)/(3k+2)$-surgeries 
on the right-hand trefoil knot $T_{2,3}$ in $S^3$ yield 
orientation-reversingly homeomorphic pairs for any non-negative integer $k$, 
i.e., the right-hand trefoil admits a chorally cosmetic surgery pairs 
along inequivalent slopes. 

After the discovery of chirally cosmetic surgery on the trefoil by Mathieu, 
Rong gave in \cite{Rong} a classification of 
Seifert knots in closed 3-manifolds (except lens spaces) 
admitting cosmetic surgeries. 
Furthermore, Matignon \cite{Matignon} gave 
a complete classification of 
non-hyperbolic knots in lens spaces admitting cosmetic surgeries. 
We remark that the cosmetic surgeries on such knots are all chirally cosmetic. 

Concerning hyperbolic knots, 
Bleiler, Hodgson and Weeks found such an example in \cite{BleilerHodgsonWeeks}]: 
They showed that there exists a hyperbolic knot which admits 
a pair of surgeries along inequivalent slopes 
yielding oppositely oriented lens spaces; 
$L(49,-19) \leftrightarrow L(49,-18)$ (mirror images). 
It was then announced in \cite{Matignon} by Matignon (preprint) 
that there are infinite family extended above.

On the other hand, the following two theorem seems to show that 
to find cosmetic surgeries are quit hard thing: 

Lackenby \cite{Lackenby}: 
Let $K$ be a homotopically trivial knot 
with irreducible, atoroidal exterior 
in 3-manifold with $\beta_1 >0$. 
Suppose that at least one of the slopes $r$, $r'$ has 
a sufficiently high distance with the meridian. 
Then $K(r)$ and $K(r')$ are orientation 
preserving homeomorphic if and only if $r = r'$,	
and	are orientation reversing homeomorphic 
if and only if $K$ is amphicheiral and $r =-r'$.

Bleiler-Hodgson-Weeks \cite{BleilerHodgsonWeeks}: 
For a hyperbolic knot $K$, 
there exists a finite set of slopes $E$ 
such that if $r$, $r'$ are distinct slopes outside $E$,
$K(r)$ and $K(r')$ homeomorphic implies that 
there exists an orientation reversing isometry $h$ such that $h(r) = r'$. 
In particular, if $K$ is a knot in $S^3$, then $K$ is amphicheiral and $r = -r'$.

\section{Proof}

We begin with recalling basic definitions and terminology about Dehn surgery. 
See \cite{Rolfsen} in details for example. 

A Dehn surgery is the following operation 
for a given knot $K$ (i.e., an embedded circle) in a 3-manifold $M$. 
First to take the exterior $E(K)$ of $K$ 
(i.e., the complement of an open tubular neighborhood of $K$ in $M$), 
and then, glue a solid torus $V$ to $E(K)$. 
Let $\gamma$ be the slope 
(i.e., an isotopy class of non-trivial unoriented simple loop) 
on the peripheral torus of $K$ in $M$ 
which is represented by the curve 
identified with the meridian of the attached solid torus via the surgery. 
Then, by $K(\gamma)$, we denote 
the manifold which obtained by the Dehn surgery on $K$, 
and call it the 3-manifold obtained by Dehn surgery on $K$ along $\gamma$. 
In particular, the Dehn surgery on $K$ along the meridional slope 
is called a \textit{trivial} Dehn surgery. 

When $K$ is a knot in $S^3$, 
by using the standard meridian-longitude system, 
slopes on the peripheral torus are parametrized by 
rational numbers with $1/0$. 
Thus, when a slope $\gamma$ corresponds to a rational number $r$, 
we use $K(r)$ in stead of $K(\gamma)$.

\bigskip

Now let us start to prove our theorem. 
Let $K$ be the knot in $S^3$ labeled as $9_{27}$ in Rolfsen's knot table. 
We then consider the two surgered manifold $K(-10/3)$ and $K(10/3)$, 
and show that they are not homeomorphic. 

The first key claim is the following. 

\begin{Claim}
If $p$ is even, then the surgered manifold $K(p/q)$ 
contains a closed non-orientable surface. \qed
\end{Claim}

This immediately follows from the result obtained in \cite{Rubinstein}, 
and actually claimed and used in \cite{Bartolini12}. 
Thus our pair of manifold $K(-10/3)$ and $K(10/3)$ 
both contain closed non-orientable embedded surfaces. 

We now consider the minimal genus of such non-orientable surfaces. 
Here, by the genus of a non-orientable surface $F$, denoted by $g(F)$, 
we mean the number of M\"{o}bius bands 
mutually disjointly embedded in $F$. 
Also $\chi (F)  = 2 - g(F)$ holds for the Euler characteristic $\chi(F)$ for $F$. 

Among our manifold $K(-10/3)$ and $K(10/3)$, 
the following holds for $K(-10/3)$. 

\begin{Claim}
The manifold $K(-10/3)$ contains 
a closed non-orientable surface $\hat{F_1}$ of genus at most $5$. 
\end{Claim}

\begin{proof}
As demonstrated in \cite{HatcherThurston}, 
in terms of the continued fractional expansions 
for the parameter of a given two-bridge knot, 
we have an algorithm to construct 
embedded surfaces in the two-bridge knot exterior. 

By using that, we see that 
there exists a non-orientable spanning surface $F_1$ for $K$ 
of genus 4 with boundary slope $-4$, 
meaning that $F_1$ has a single boundary component 
which represents the slope corresponding to $-4$ on $\partial E(K)$. 
Also, it can be checked by the Dunfield's program \cite{Dunfield}, 
which implements the Hatcher-Thurston's algorithm. 

We here note that 
the distance $\Delta (-4, -10/3)$ between 
the pair of slopes $-4$ and $-10/3$ is calculated as 
$| -4 \cdot 3 - (-10) \cdot 1 | = 2$, 
where the \textit{distance} of slopes is defined as 
the minimal intersection number of their representatives, 
and is calculated by $| ps - qr |$ for the slopes $p/q$ and $r/s$. 
See \cite{Rolfsen} for example. 

Since $\Delta (-4, -10/3) =2$, by adding a M\"{o}bius band, 
equivalently, a converse operation of boundary-compressing, 
we have a non-orientable surface of genus 5 with boundary-slope $-10/3$. 
\end{proof}

\begin{figure}[hbt]
\unitlength 0.1in
\begin{picture}( 30.3000, 18.0000)( 29.9000,-28.0000)
%
\special{pn 8}%
\special{ar 3200 1600 210 600  0.0000000 6.2831853}%
%
\special{pn 8}%
\special{ar 5800 1600 210 600  1.5707963 1.6004260}%
\special{ar 5800 1600 210 600  1.6893148 1.7189445}%
\special{ar 5800 1600 210 600  1.8078334 1.8374630}%
\special{ar 5800 1600 210 600  1.9263519 1.9559815}%
\special{ar 5800 1600 210 600  2.0448704 2.0745000}%
\special{ar 5800 1600 210 600  2.1633889 2.1930185}%
\special{ar 5800 1600 210 600  2.2819074 2.3115371}%
\special{ar 5800 1600 210 600  2.4004260 2.4300556}%
\special{ar 5800 1600 210 600  2.5189445 2.5485741}%
\special{ar 5800 1600 210 600  2.6374630 2.6670926}%
\special{ar 5800 1600 210 600  2.7559815 2.7856111}%
\special{ar 5800 1600 210 600  2.8745000 2.9041297}%
\special{ar 5800 1600 210 600  2.9930185 3.0226482}%
\special{ar 5800 1600 210 600  3.1115371 3.1411667}%
\special{ar 5800 1600 210 600  3.2300556 3.2596852}%
\special{ar 5800 1600 210 600  3.3485741 3.3782037}%
\special{ar 5800 1600 210 600  3.4670926 3.4967223}%
\special{ar 5800 1600 210 600  3.5856111 3.6152408}%
\special{ar 5800 1600 210 600  3.7041297 3.7337593}%
\special{ar 5800 1600 210 600  3.8226482 3.8522778}%
\special{ar 5800 1600 210 600  3.9411667 3.9707963}%
\special{ar 5800 1600 210 600  4.0596852 4.0893148}%
\special{ar 5800 1600 210 600  4.1782037 4.2078334}%
\special{ar 5800 1600 210 600  4.2967223 4.3263519}%
\special{ar 5800 1600 210 600  4.4152408 4.4448704}%
\special{ar 5800 1600 210 600  4.5337593 4.5633889}%
\special{ar 5800 1600 210 600  4.6522778 4.6819074}%
%
\special{pn 8}%
\special{ar 5800 1600 220 600  4.7123890 6.2831853}%
\special{ar 5800 1600 220 600  0.0000000 1.5707963}%
%
\special{pn 8}%
\special{pa 5800 1000}%
\special{pa 3200 1000}%
\special{fp}%
%
\special{pn 13}%
\special{pa 3210 2200}%
\special{pa 3910 2200}%
\special{fp}%
%
\special{pn 13}%
\special{pa 5280 2200}%
\special{pa 5800 2200}%
\special{fp}%
%
\special{pn 13}%
\special{ar 4410 2200 490 1200  5.1191949 5.1333961}%
\special{ar 4410 2200 490 1200  5.1759996 5.1902008}%
\special{ar 4410 2200 490 1200  5.2328044 5.2470056}%
\special{ar 4410 2200 490 1200  5.2896091 5.3038103}%
\special{ar 4410 2200 490 1200  5.3464138 5.3606150}%
\special{ar 4410 2200 490 1200  5.4032186 5.4174198}%
\special{ar 4410 2200 490 1200  5.4600233 5.4742245}%
\special{ar 4410 2200 490 1200  5.5168280 5.5310292}%
\special{ar 4410 2200 490 1200  5.5736328 5.5878340}%
\special{ar 4410 2200 490 1200  5.6304375 5.6446387}%
\special{ar 4410 2200 490 1200  5.6872422 5.7014434}%
\special{ar 4410 2200 490 1200  5.7440470 5.7582482}%
\special{ar 4410 2200 490 1200  5.8008517 5.8150529}%
\special{ar 4410 2200 490 1200  5.8576564 5.8718576}%
\special{ar 4410 2200 490 1200  5.9144612 5.9286624}%
\special{ar 4410 2200 490 1200  5.9712659 5.9854671}%
\special{ar 4410 2200 490 1200  6.0280706 6.0422718}%
\special{ar 4410 2200 490 1200  6.0848754 6.0990766}%
\special{ar 4410 2200 490 1200  6.1416801 6.1558813}%
\special{ar 4410 2200 490 1200  6.1984848 6.2126860}%
\special{ar 4410 2200 490 1200  6.2552896 6.2694908}%
%
\special{pn 13}%
\special{ar 4810 2200 500 1200  3.1415927 4.7123890}%
%
\special{pn 13}%
\special{ar 4800 2200 490 1200  4.7123890 4.7265902}%
\special{ar 4800 2200 490 1200  4.7691937 4.7833949}%
\special{ar 4800 2200 490 1200  4.8259984 4.8401996}%
\special{ar 4800 2200 490 1200  4.8828032 4.8970044}%
\special{ar 4800 2200 490 1200  4.9396079 4.9538091}%
\special{ar 4800 2200 490 1200  4.9964126 5.0106138}%
\special{ar 4800 2200 490 1200  5.0532174 5.0674186}%
\special{ar 4800 2200 490 1200  5.1100221 5.1242233}%
\special{ar 4800 2200 490 1200  5.1668269 5.1810280}%
\special{ar 4800 2200 490 1200  5.2236316 5.2378328}%
\special{ar 4800 2200 490 1200  5.2804363 5.2946375}%
\special{ar 4800 2200 490 1200  5.3372411 5.3514422}%
\special{ar 4800 2200 490 1200  5.3940458 5.4082470}%
\special{ar 4800 2200 490 1200  5.4508505 5.4650517}%
\special{ar 4800 2200 490 1200  5.5076553 5.5218564}%
\special{ar 4800 2200 490 1200  5.5644600 5.5786612}%
\special{ar 4800 2200 490 1200  5.6212647 5.6354659}%
\special{ar 4800 2200 490 1200  5.6780695 5.6922706}%
\special{ar 4800 2200 490 1200  5.7348742 5.7490754}%
\special{ar 4800 2200 490 1200  5.7916789 5.8058801}%
\special{ar 4800 2200 490 1200  5.8484837 5.8626848}%
\special{ar 4800 2200 490 1200  5.9052884 5.9194896}%
\special{ar 4800 2200 490 1200  5.9620931 5.9762943}%
\special{ar 4800 2200 490 1200  6.0188979 6.0330990}%
\special{ar 4800 2200 490 1200  6.0757026 6.0899038}%
\special{ar 4800 2200 490 1200  6.1325073 6.1467085}%
\special{ar 4800 2200 490 1200  6.1893121 6.2035132}%
\special{ar 4800 2200 490 1200  6.2461168 6.2603180}%
%
\special{pn 13}%
\special{ar 4810 2200 500 1200  3.1415927 4.7123890}%
%
\special{pn 13}%
\special{ar 4410 2200 500 1200  3.1415927 4.7123890}%
%
\special{pn 13}%
\special{ar 4410 2200 500 1200  3.1415927 4.7123890}%
%
\special{pn 13}%
\special{pa 4306 2200}%
\special{pa 4896 2200}%
\special{fp}%
%
\special{pn 8}%
\special{pa 3200 2200}%
\special{pa 3200 2800}%
\special{fp}%
\special{pa 3200 2800}%
\special{pa 5800 2800}%
\special{fp}%
\special{pa 5800 2800}%
\special{pa 5800 2200}%
\special{fp}%
%
\special{pn 13}%
\special{pa 3890 2200}%
\special{pa 5290 2200}%
\special{ip}%
%
\special{pn 4}%
\special{pa 4450 1370}%
\special{pa 3950 1870}%
\special{fp}%
\special{pa 4420 1460}%
\special{pa 3940 1940}%
\special{fp}%
\special{pa 4400 1540}%
\special{pa 3930 2010}%
\special{fp}%
\special{pa 4380 1620}%
\special{pa 3930 2070}%
\special{fp}%
\special{pa 4360 1700}%
\special{pa 3930 2130}%
\special{fp}%
\special{pa 4350 1770}%
\special{pa 3940 2180}%
\special{fp}%
\special{pa 4340 1840}%
\special{pa 3990 2190}%
\special{fp}%
\special{pa 4330 1910}%
\special{pa 4050 2190}%
\special{fp}%
\special{pa 4320 1980}%
\special{pa 4110 2190}%
\special{fp}%
\special{pa 4320 2040}%
\special{pa 4170 2190}%
\special{fp}%
\special{pa 4320 2100}%
\special{pa 4230 2190}%
\special{fp}%
\special{pa 4320 2160}%
\special{pa 4290 2190}%
\special{fp}%
\special{pa 4510 1250}%
\special{pa 3960 1800}%
\special{fp}%
\special{pa 4590 1110}%
\special{pa 3970 1730}%
\special{fp}%
\special{pa 4560 1080}%
\special{pa 3980 1660}%
\special{fp}%
\special{pa 4520 1060}%
\special{pa 4000 1580}%
\special{fp}%
\special{pa 4490 1030}%
\special{pa 4020 1500}%
\special{fp}%
\special{pa 4440 1020}%
\special{pa 4050 1410}%
\special{fp}%
\special{pa 4380 1020}%
\special{pa 4100 1300}%
\special{fp}%
%
\special{pn 4}%
\special{pa 4640 1000}%
\special{pa 4570 1070}%
\special{fp}%
\special{pa 4700 1000}%
\special{pa 4610 1090}%
\special{fp}%
\special{pa 4580 1000}%
\special{pa 4540 1040}%
\special{fp}%
%
\special{pn 4}%
\special{pa 4260 1380}%
\special{pa 4230 1350}%
\special{fp}%
\special{pa 4230 1410}%
\special{pa 4200 1380}%
\special{fp}%
\special{pa 4200 1440}%
\special{pa 4170 1410}%
\special{fp}%
\special{pa 4170 1470}%
\special{pa 4140 1440}%
\special{fp}%
\special{pa 4140 1500}%
\special{pa 4110 1470}%
\special{fp}%
\special{pa 4110 1530}%
\special{pa 4080 1500}%
\special{fp}%
\special{pa 4080 1560}%
\special{pa 4050 1530}%
\special{fp}%
\special{pa 4050 1590}%
\special{pa 4020 1560}%
\special{fp}%
\special{pa 4020 1620}%
\special{pa 4000 1600}%
\special{fp}%
\special{pa 4290 1350}%
\special{pa 4260 1320}%
\special{fp}%
\special{pa 4320 1320}%
\special{pa 4290 1290}%
\special{fp}%
\special{pa 4350 1290}%
\special{pa 4320 1260}%
\special{fp}%
\special{pa 4380 1260}%
\special{pa 4350 1230}%
\special{fp}%
\special{pa 4410 1230}%
\special{pa 4380 1200}%
\special{fp}%
\special{pa 4440 1200}%
\special{pa 4410 1170}%
\special{fp}%
\special{pa 4470 1170}%
\special{pa 4440 1140}%
\special{fp}%
\special{pa 4500 1140}%
\special{pa 4470 1110}%
\special{fp}%
\special{pa 4530 1110}%
\special{pa 4500 1080}%
\special{fp}%
%
\special{pn 4}%
\special{pa 4260 1320}%
\special{pa 4230 1290}%
\special{fp}%
\special{pa 4230 1350}%
\special{pa 4200 1320}%
\special{fp}%
\special{pa 4200 1380}%
\special{pa 4170 1350}%
\special{fp}%
\special{pa 4170 1410}%
\special{pa 4140 1380}%
\special{fp}%
\special{pa 4140 1440}%
\special{pa 4110 1410}%
\special{fp}%
\special{pa 4110 1470}%
\special{pa 4080 1440}%
\special{fp}%
\special{pa 4080 1500}%
\special{pa 4050 1470}%
\special{fp}%
\special{pa 4050 1530}%
\special{pa 4020 1500}%
\special{fp}%
\special{pa 4290 1290}%
\special{pa 4260 1260}%
\special{fp}%
\special{pa 4320 1260}%
\special{pa 4290 1230}%
\special{fp}%
\special{pa 4350 1230}%
\special{pa 4320 1200}%
\special{fp}%
\special{pa 4380 1200}%
\special{pa 4350 1170}%
\special{fp}%
\special{pa 4410 1170}%
\special{pa 4380 1140}%
\special{fp}%
\special{pa 4440 1140}%
\special{pa 4410 1110}%
\special{fp}%
\special{pa 4470 1110}%
\special{pa 4440 1080}%
\special{fp}%
\special{pa 4500 1080}%
\special{pa 4470 1050}%
\special{fp}%
%
\special{pn 4}%
\special{pa 4290 1230}%
\special{pa 4260 1200}%
\special{fp}%
\special{pa 4260 1260}%
\special{pa 4230 1230}%
\special{fp}%
\special{pa 4230 1290}%
\special{pa 4200 1260}%
\special{fp}%
\special{pa 4200 1320}%
\special{pa 4170 1290}%
\special{fp}%
\special{pa 4170 1350}%
\special{pa 4140 1320}%
\special{fp}%
\special{pa 4140 1380}%
\special{pa 4110 1350}%
\special{fp}%
\special{pa 4110 1410}%
\special{pa 4080 1380}%
\special{fp}%
\special{pa 4080 1440}%
\special{pa 4050 1410}%
\special{fp}%
\special{pa 4320 1200}%
\special{pa 4290 1170}%
\special{fp}%
\special{pa 4350 1170}%
\special{pa 4320 1140}%
\special{fp}%
\special{pa 4380 1140}%
\special{pa 4350 1110}%
\special{fp}%
\special{pa 4410 1110}%
\special{pa 4380 1080}%
\special{fp}%
\special{pa 4440 1080}%
\special{pa 4410 1050}%
\special{fp}%
\special{pa 4470 1050}%
\special{pa 4440 1020}%
\special{fp}%
%
\special{pn 4}%
\special{pa 4260 1200}%
\special{pa 4230 1170}%
\special{fp}%
\special{pa 4230 1230}%
\special{pa 4200 1200}%
\special{fp}%
\special{pa 4200 1260}%
\special{pa 4170 1230}%
\special{fp}%
\special{pa 4170 1290}%
\special{pa 4140 1260}%
\special{fp}%
\special{pa 4140 1320}%
\special{pa 4110 1290}%
\special{fp}%
\special{pa 4110 1350}%
\special{pa 4090 1330}%
\special{fp}%
\special{pa 4290 1170}%
\special{pa 4260 1140}%
\special{fp}%
\special{pa 4320 1140}%
\special{pa 4290 1110}%
\special{fp}%
\special{pa 4350 1110}%
\special{pa 4320 1080}%
\special{fp}%
\special{pa 4380 1080}%
\special{pa 4350 1050}%
\special{fp}%
\special{pa 4410 1050}%
\special{pa 4380 1020}%
\special{fp}%
%
\special{pn 4}%
\special{pa 4260 1140}%
\special{pa 4230 1110}%
\special{fp}%
\special{pa 4230 1170}%
\special{pa 4200 1140}%
\special{fp}%
\special{pa 4200 1200}%
\special{pa 4180 1180}%
\special{fp}%
\special{pa 4170 1230}%
\special{pa 4150 1210}%
\special{fp}%
\special{pa 4290 1110}%
\special{pa 4260 1080}%
\special{fp}%
\special{pa 4320 1080}%
\special{pa 4300 1060}%
\special{fp}%
%
\special{pn 4}%
\special{pa 4200 1500}%
\special{pa 4170 1470}%
\special{fp}%
\special{pa 4170 1530}%
\special{pa 4140 1500}%
\special{fp}%
\special{pa 4140 1560}%
\special{pa 4110 1530}%
\special{fp}%
\special{pa 4110 1590}%
\special{pa 4080 1560}%
\special{fp}%
\special{pa 4080 1620}%
\special{pa 4050 1590}%
\special{fp}%
\special{pa 4050 1650}%
\special{pa 4020 1620}%
\special{fp}%
\special{pa 4020 1680}%
\special{pa 3990 1650}%
\special{fp}%
\special{pa 3990 1710}%
\special{pa 3970 1690}%
\special{fp}%
\special{pa 4230 1470}%
\special{pa 4200 1440}%
\special{fp}%
\special{pa 4260 1440}%
\special{pa 4230 1410}%
\special{fp}%
\special{pa 4290 1410}%
\special{pa 4260 1380}%
\special{fp}%
\special{pa 4320 1380}%
\special{pa 4290 1350}%
\special{fp}%
\special{pa 4350 1350}%
\special{pa 4320 1320}%
\special{fp}%
\special{pa 4380 1320}%
\special{pa 4350 1290}%
\special{fp}%
\special{pa 4410 1290}%
\special{pa 4380 1260}%
\special{fp}%
\special{pa 4440 1260}%
\special{pa 4410 1230}%
\special{fp}%
\special{pa 4470 1230}%
\special{pa 4440 1200}%
\special{fp}%
\special{pa 4500 1200}%
\special{pa 4470 1170}%
\special{fp}%
\special{pa 4530 1170}%
\special{pa 4500 1140}%
\special{fp}%
\special{pa 4560 1140}%
\special{pa 4530 1110}%
\special{fp}%
%
\special{pn 4}%
\special{pa 4170 1590}%
\special{pa 4140 1560}%
\special{fp}%
\special{pa 4140 1620}%
\special{pa 4110 1590}%
\special{fp}%
\special{pa 4110 1650}%
\special{pa 4080 1620}%
\special{fp}%
\special{pa 4080 1680}%
\special{pa 4050 1650}%
\special{fp}%
\special{pa 4050 1710}%
\special{pa 4020 1680}%
\special{fp}%
\special{pa 4020 1740}%
\special{pa 3990 1710}%
\special{fp}%
\special{pa 3990 1770}%
\special{pa 3960 1740}%
\special{fp}%
\special{pa 4200 1560}%
\special{pa 4170 1530}%
\special{fp}%
\special{pa 4230 1530}%
\special{pa 4200 1500}%
\special{fp}%
\special{pa 4260 1500}%
\special{pa 4230 1470}%
\special{fp}%
\special{pa 4290 1470}%
\special{pa 4260 1440}%
\special{fp}%
\special{pa 4320 1440}%
\special{pa 4290 1410}%
\special{fp}%
\special{pa 4350 1410}%
\special{pa 4320 1380}%
\special{fp}%
\special{pa 4380 1380}%
\special{pa 4350 1350}%
\special{fp}%
\special{pa 4410 1350}%
\special{pa 4380 1320}%
\special{fp}%
\special{pa 4440 1320}%
\special{pa 4410 1290}%
\special{fp}%
\special{pa 4470 1290}%
\special{pa 4440 1260}%
\special{fp}%
\special{pa 4500 1260}%
\special{pa 4470 1230}%
\special{fp}%
\special{pa 4520 1220}%
\special{pa 4500 1200}%
\special{fp}%
\special{pa 4550 1190}%
\special{pa 4530 1170}%
\special{fp}%
%
\special{pn 4}%
\special{pa 4080 1740}%
\special{pa 4050 1710}%
\special{fp}%
\special{pa 4050 1770}%
\special{pa 4020 1740}%
\special{fp}%
\special{pa 4020 1800}%
\special{pa 3990 1770}%
\special{fp}%
\special{pa 3990 1830}%
\special{pa 3960 1800}%
\special{fp}%
\special{pa 4110 1710}%
\special{pa 4080 1680}%
\special{fp}%
\special{pa 4140 1680}%
\special{pa 4110 1650}%
\special{fp}%
\special{pa 4170 1650}%
\special{pa 4140 1620}%
\special{fp}%
\special{pa 4200 1620}%
\special{pa 4170 1590}%
\special{fp}%
\special{pa 4230 1590}%
\special{pa 4200 1560}%
\special{fp}%
\special{pa 4260 1560}%
\special{pa 4230 1530}%
\special{fp}%
\special{pa 4290 1530}%
\special{pa 4260 1500}%
\special{fp}%
\special{pa 4320 1500}%
\special{pa 4290 1470}%
\special{fp}%
\special{pa 4350 1470}%
\special{pa 4320 1440}%
\special{fp}%
\special{pa 4380 1440}%
\special{pa 4350 1410}%
\special{fp}%
\special{pa 4410 1410}%
\special{pa 4380 1380}%
\special{fp}%
\special{pa 4440 1380}%
\special{pa 4410 1350}%
\special{fp}%
\special{pa 4460 1340}%
\special{pa 4440 1320}%
\special{fp}%
%
\special{pn 4}%
\special{pa 4200 1680}%
\special{pa 4170 1650}%
\special{fp}%
\special{pa 4170 1710}%
\special{pa 4140 1680}%
\special{fp}%
\special{pa 4140 1740}%
\special{pa 4110 1710}%
\special{fp}%
\special{pa 4110 1770}%
\special{pa 4080 1740}%
\special{fp}%
\special{pa 4080 1800}%
\special{pa 4050 1770}%
\special{fp}%
\special{pa 4050 1830}%
\special{pa 4020 1800}%
\special{fp}%
\special{pa 4020 1860}%
\special{pa 3990 1830}%
\special{fp}%
\special{pa 3990 1890}%
\special{pa 3960 1860}%
\special{fp}%
\special{pa 3960 1920}%
\special{pa 3940 1900}%
\special{fp}%
\special{pa 4230 1650}%
\special{pa 4200 1620}%
\special{fp}%
\special{pa 4260 1620}%
\special{pa 4230 1590}%
\special{fp}%
\special{pa 4290 1590}%
\special{pa 4260 1560}%
\special{fp}%
\special{pa 4320 1560}%
\special{pa 4290 1530}%
\special{fp}%
\special{pa 4350 1530}%
\special{pa 4320 1500}%
\special{fp}%
\special{pa 4380 1500}%
\special{pa 4350 1470}%
\special{fp}%
\special{pa 4410 1470}%
\special{pa 4380 1440}%
\special{fp}%
\special{pa 4430 1430}%
\special{pa 4410 1410}%
\special{fp}%
%
\special{pn 4}%
\special{pa 4200 1740}%
\special{pa 4170 1710}%
\special{fp}%
\special{pa 4170 1770}%
\special{pa 4140 1740}%
\special{fp}%
\special{pa 4140 1800}%
\special{pa 4110 1770}%
\special{fp}%
\special{pa 4110 1830}%
\special{pa 4080 1800}%
\special{fp}%
\special{pa 4080 1860}%
\special{pa 4050 1830}%
\special{fp}%
\special{pa 4050 1890}%
\special{pa 4020 1860}%
\special{fp}%
\special{pa 4020 1920}%
\special{pa 3990 1890}%
\special{fp}%
\special{pa 3990 1950}%
\special{pa 3960 1920}%
\special{fp}%
\special{pa 3960 1980}%
\special{pa 3940 1960}%
\special{fp}%
\special{pa 4230 1710}%
\special{pa 4200 1680}%
\special{fp}%
\special{pa 4260 1680}%
\special{pa 4230 1650}%
\special{fp}%
\special{pa 4290 1650}%
\special{pa 4260 1620}%
\special{fp}%
\special{pa 4320 1620}%
\special{pa 4290 1590}%
\special{fp}%
\special{pa 4350 1590}%
\special{pa 4320 1560}%
\special{fp}%
\special{pa 4380 1560}%
\special{pa 4350 1530}%
\special{fp}%
\special{pa 4400 1520}%
\special{pa 4380 1500}%
\special{fp}%
%
\special{pn 4}%
\special{pa 4110 1890}%
\special{pa 4080 1860}%
\special{fp}%
\special{pa 4080 1920}%
\special{pa 4050 1890}%
\special{fp}%
\special{pa 4050 1950}%
\special{pa 4020 1920}%
\special{fp}%
\special{pa 4020 1980}%
\special{pa 3990 1950}%
\special{fp}%
\special{pa 3990 2010}%
\special{pa 3960 1980}%
\special{fp}%
\special{pa 3960 2040}%
\special{pa 3930 2010}%
\special{fp}%
\special{pa 4140 1860}%
\special{pa 4110 1830}%
\special{fp}%
\special{pa 4170 1830}%
\special{pa 4140 1800}%
\special{fp}%
\special{pa 4200 1800}%
\special{pa 4170 1770}%
\special{fp}%
\special{pa 4230 1770}%
\special{pa 4200 1740}%
\special{fp}%
\special{pa 4260 1740}%
\special{pa 4230 1710}%
\special{fp}%
\special{pa 4290 1710}%
\special{pa 4260 1680}%
\special{fp}%
\special{pa 4320 1680}%
\special{pa 4290 1650}%
\special{fp}%
\special{pa 4350 1650}%
\special{pa 4320 1620}%
\special{fp}%
\special{pa 4380 1620}%
\special{pa 4350 1590}%
\special{fp}%
%
\special{pn 4}%
\special{pa 4200 1860}%
\special{pa 4170 1830}%
\special{fp}%
\special{pa 4170 1890}%
\special{pa 4140 1860}%
\special{fp}%
\special{pa 4140 1920}%
\special{pa 4110 1890}%
\special{fp}%
\special{pa 4110 1950}%
\special{pa 4080 1920}%
\special{fp}%
\special{pa 4080 1980}%
\special{pa 4050 1950}%
\special{fp}%
\special{pa 4050 2010}%
\special{pa 4020 1980}%
\special{fp}%
\special{pa 4020 2040}%
\special{pa 3990 2010}%
\special{fp}%
\special{pa 3990 2070}%
\special{pa 3960 2040}%
\special{fp}%
\special{pa 3960 2100}%
\special{pa 3930 2070}%
\special{fp}%
\special{pa 4230 1830}%
\special{pa 4200 1800}%
\special{fp}%
\special{pa 4260 1800}%
\special{pa 4230 1770}%
\special{fp}%
\special{pa 4290 1770}%
\special{pa 4260 1740}%
\special{fp}%
\special{pa 4320 1740}%
\special{pa 4290 1710}%
\special{fp}%
\special{pa 4350 1710}%
\special{pa 4320 1680}%
\special{fp}%
\special{pa 4370 1670}%
\special{pa 4350 1650}%
\special{fp}%
%
\special{pn 4}%
\special{pa 4140 1980}%
\special{pa 4110 1950}%
\special{fp}%
\special{pa 4110 2010}%
\special{pa 4080 1980}%
\special{fp}%
\special{pa 4080 2040}%
\special{pa 4050 2010}%
\special{fp}%
\special{pa 4050 2070}%
\special{pa 4020 2040}%
\special{fp}%
\special{pa 4020 2100}%
\special{pa 3990 2070}%
\special{fp}%
\special{pa 3990 2130}%
\special{pa 3960 2100}%
\special{fp}%
\special{pa 3960 2160}%
\special{pa 3930 2130}%
\special{fp}%
\special{pa 4170 1950}%
\special{pa 4140 1920}%
\special{fp}%
\special{pa 4200 1920}%
\special{pa 4170 1890}%
\special{fp}%
\special{pa 4230 1890}%
\special{pa 4200 1860}%
\special{fp}%
\special{pa 4260 1860}%
\special{pa 4230 1830}%
\special{fp}%
\special{pa 4290 1830}%
\special{pa 4260 1800}%
\special{fp}%
\special{pa 4320 1800}%
\special{pa 4290 1770}%
\special{fp}%
\special{pa 4350 1770}%
\special{pa 4320 1740}%
\special{fp}%
%
\special{pn 4}%
\special{pa 4200 1980}%
\special{pa 4170 1950}%
\special{fp}%
\special{pa 4170 2010}%
\special{pa 4140 1980}%
\special{fp}%
\special{pa 4140 2040}%
\special{pa 4110 2010}%
\special{fp}%
\special{pa 4110 2070}%
\special{pa 4080 2040}%
\special{fp}%
\special{pa 4080 2100}%
\special{pa 4050 2070}%
\special{fp}%
\special{pa 4050 2130}%
\special{pa 4020 2100}%
\special{fp}%
\special{pa 4020 2160}%
\special{pa 3990 2130}%
\special{fp}%
\special{pa 3990 2190}%
\special{pa 3960 2160}%
\special{fp}%
\special{pa 4230 1950}%
\special{pa 4200 1920}%
\special{fp}%
\special{pa 4260 1920}%
\special{pa 4230 1890}%
\special{fp}%
\special{pa 4290 1890}%
\special{pa 4260 1860}%
\special{fp}%
\special{pa 4320 1860}%
\special{pa 4290 1830}%
\special{fp}%
\special{pa 4340 1820}%
\special{pa 4320 1800}%
\special{fp}%
%
\special{pn 4}%
\special{pa 4230 2010}%
\special{pa 4200 1980}%
\special{fp}%
\special{pa 4200 2040}%
\special{pa 4170 2010}%
\special{fp}%
\special{pa 4170 2070}%
\special{pa 4140 2040}%
\special{fp}%
\special{pa 4140 2100}%
\special{pa 4110 2070}%
\special{fp}%
\special{pa 4110 2130}%
\special{pa 4080 2100}%
\special{fp}%
\special{pa 4080 2160}%
\special{pa 4050 2130}%
\special{fp}%
\special{pa 4050 2190}%
\special{pa 4020 2160}%
\special{fp}%
\special{pa 4260 1980}%
\special{pa 4230 1950}%
\special{fp}%
\special{pa 4290 1950}%
\special{pa 4260 1920}%
\special{fp}%
\special{pa 4320 1920}%
\special{pa 4290 1890}%
\special{fp}%
%
\special{pn 4}%
\special{pa 4260 2040}%
\special{pa 4230 2010}%
\special{fp}%
\special{pa 4230 2070}%
\special{pa 4200 2040}%
\special{fp}%
\special{pa 4200 2100}%
\special{pa 4170 2070}%
\special{fp}%
\special{pa 4170 2130}%
\special{pa 4140 2100}%
\special{fp}%
\special{pa 4140 2160}%
\special{pa 4110 2130}%
\special{fp}%
\special{pa 4110 2190}%
\special{pa 4080 2160}%
\special{fp}%
\special{pa 4290 2010}%
\special{pa 4260 1980}%
\special{fp}%
\special{pa 4320 1980}%
\special{pa 4290 1950}%
\special{fp}%
%
\special{pn 4}%
\special{pa 4260 2100}%
\special{pa 4230 2070}%
\special{fp}%
\special{pa 4230 2130}%
\special{pa 4200 2100}%
\special{fp}%
\special{pa 4200 2160}%
\special{pa 4170 2130}%
\special{fp}%
\special{pa 4170 2190}%
\special{pa 4140 2160}%
\special{fp}%
\special{pa 4290 2070}%
\special{pa 4260 2040}%
\special{fp}%
\special{pa 4320 2040}%
\special{pa 4290 2010}%
\special{fp}%
%
\special{pn 4}%
\special{pa 4290 2130}%
\special{pa 4260 2100}%
\special{fp}%
\special{pa 4260 2160}%
\special{pa 4230 2130}%
\special{fp}%
\special{pa 4230 2190}%
\special{pa 4200 2160}%
\special{fp}%
\special{pa 4320 2100}%
\special{pa 4290 2070}%
\special{fp}%
%
\special{pn 4}%
\special{pa 4290 2190}%
\special{pa 4260 2160}%
\special{fp}%
\special{pa 4320 2160}%
\special{pa 4290 2130}%
\special{fp}%
%
\special{pn 4}%
\special{pa 4590 1050}%
\special{pa 4560 1020}%
\special{fp}%
%
\special{pn 4}%
\special{pa 4620 1080}%
\special{pa 4590 1050}%
\special{fp}%
\special{pa 4650 1050}%
\special{pa 4620 1020}%
\special{fp}%
%
\special{pn 4}%
\special{pa 5280 1980}%
\special{pa 4820 1520}%
\special{fp}%
\special{pa 5270 1910}%
\special{pa 4790 1430}%
\special{fp}%
\special{pa 5260 1840}%
\special{pa 4750 1330}%
\special{fp}%
\special{pa 5250 1770}%
\special{pa 4690 1210}%
\special{fp}%
\special{pa 5240 1700}%
\special{pa 4640 1100}%
\special{fp}%
\special{pa 5220 1620}%
\special{pa 4670 1070}%
\special{fp}%
\special{pa 5200 1540}%
\special{pa 4710 1050}%
\special{fp}%
\special{pa 5170 1450}%
\special{pa 4750 1030}%
\special{fp}%
\special{pa 5140 1360}%
\special{pa 4800 1020}%
\special{fp}%
\special{pa 5080 1240}%
\special{pa 4870 1030}%
\special{fp}%
\special{pa 5280 2040}%
\special{pa 4840 1600}%
\special{fp}%
\special{pa 5280 2100}%
\special{pa 4850 1670}%
\special{fp}%
\special{pa 5280 2160}%
\special{pa 4870 1750}%
\special{fp}%
\special{pa 5250 2190}%
\special{pa 4880 1820}%
\special{fp}%
\special{pa 5190 2190}%
\special{pa 4890 1890}%
\special{fp}%
\special{pa 5130 2190}%
\special{pa 4900 1960}%
\special{fp}%
\special{pa 5070 2190}%
\special{pa 4900 2020}%
\special{fp}%
\special{pa 5010 2190}%
\special{pa 4900 2080}%
\special{fp}%
\special{pa 4950 2190}%
\special{pa 4900 2140}%
\special{fp}%
%
\special{pn 4}%
\special{pa 4550 1030}%
\special{pa 4520 1060}%
\special{fp}%
\special{pa 4520 1000}%
\special{pa 4490 1030}%
\special{fp}%
\special{pa 4580 1060}%
\special{pa 4560 1080}%
\special{fp}%
%
\special{pn 4}%
\special{pa 4560 1080}%
\special{pa 4530 1050}%
\special{fp}%
%
\special{pn 4}%
\special{pa 4530 1050}%
\special{pa 4500 1020}%
\special{fp}%
\end{picture}%
\caption{M\"{o}bius band attaching}
\end{figure}
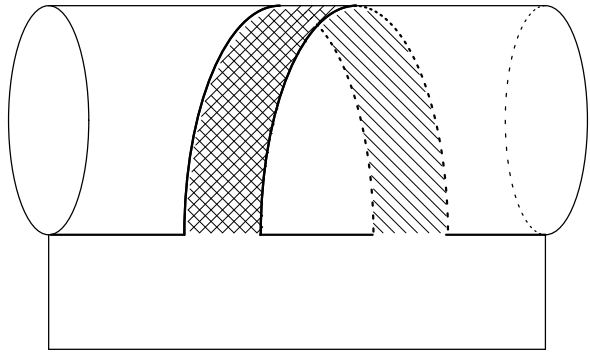

On the other hand, the following holds for $K(10/3)$. 

\begin{Claim}
The manifold $K(10/3)$ does not contain 
closed non-orientable surfaces of genus at most $5$. 
\end{Claim}

\begin{proof}
We suppose that $K(10/3)$ contains 
a closed non-orientable surface $\hat{F_2}$ of genus at most 5, 
and will find a contradiction. 
After compressions, if necessary, 
we can assume that $\hat{F_2}$ is incompressible. 

Then, as shown by Przytycki in \cite[Proposition 3.3]{Przytycki}, 
$\hat{F_2}$ can be isotoped so that 
$F_2 = \hat{F_2} \cap E(K)$ is incompressible, boundary-incompressible, 
and not boundary-parallel properly embedded in $E(K)$, 
and $\hat{F_2} \cap V$ is incompressible in the attached solid torus $V$. 

For the candidates of $\hat{F_2}$, 
by using the Dunfield's program, 
we can verify that 
there are exactly 8 such surfaces in $E(K)$, 
and their genera are at least $4$. 
Now, since we are assuming $\hat{F_2}$ of genus at most 5, 
it implies that the genus $g(F_2)$ of $F_2$ must be either $4$ or $5$.  

Consider the case where $g(F_2) = 5$. 
In this case, the Dunfield's program tells us that 
their boundary-slopes are either $-2$, $2$, $6$, or $10$. 
However, since $g(\hat{F_2})=g(F_2)=5$, 
it follows that $\hat{F_2} - F_2$ must be a disk,  
which implies the boundary-slope of $F_2$ must be $10/3$. 
A contradiction occurs. 
 
Consider the case where $g(F_2) = 4$. 
In this case, the Dunfield's program tells us that 
their boundary-slopes are either $-8$, $-4$, or $0$. 
Now, since $g(\hat{F_2})=5$ and $g(F_2)=4$, 
$\hat{F_2} \cap V$ gives a M\"{o}bius band $M$ 
properly embedded in the attached solid torus $V$. 

Also as shown in \cite{Przytycki}, 
this $M$ must be boundary compressible in $V$. 
Then, by single boundary-compression on $M$ in $V$, 
we have 
a non-orientable incompressible, boundary-compressible surface $F_2'$ 
properly embedded in $E(K)$ with boundary-slope $10/3$. 
This implies that $\Delta ( r_2 , 10/3 )=2$ must hold 
for the boundary slope $r_2$ of $F_2$. 

However, since $r_2$ must be either $-8$, $-4$, or $0$, 
it follows that $\Delta ( r_2 , 10/3 ) \ne 2$. 
Again a contradiction occurs. 
\end{proof}

These claims show that 
the pair of manifolds $K(-10/3)$ and $K(10/3)$ 
are not homeomorphic.

\bibliographystyle{amsplain}

\end{document}